\documentclass[12pt,a4paper]{article}
\usepackage{theorem}
\newtheorem{lemma}{Lemma}

\newtheorem{theorem}[lemma]{Theorem}

\newtheorem{exampl}[lemma]{Example}

\newcommand{\R}{{\bf R}}

\newcommand{\Z}{{\bf Z}}

\newcommand{\alp}{\alpha}
\newcommand{\bet}{\beta}
\newcommand{\gam}{\gamma}

\newcommand{\del}{\delta}

\newcommand{\pr}{\prime}

\newcommand{\eqref}[1]{(\ref{#1})}

\newenvironment{proof}{\textbf{Proof}}{\hfill$\Box$}

\newcommand{\txtfrac}[2]{\mbox{$\frac{#1}{#2}$}}
\usepackage{amssymb}
\usepackage{graphics}
\usepackage{graphpap}
\usepackage{hyperref}
\title{Archimedes' calculations of square roots}
\author{E. B. Davies}
\date{3 January 2011}
\begin{document}
\maketitle
\begin{abstract}
We reconsider Archimedes evaluations of several square roots in `Measurement of a Circle'. We show that several methods proposed over the last century or so for his evaluations fail one or more criteria of plausibility. We also provide internal evidence that he probably used an interpolation technique. The conclusions are relevant to the precise calculations by which he obtained upper and lower bounds on $\pi$.
\end{abstract}

2010MSC classification: 01A20

\section{Introduction}
We do not know how Archimedes proved his famous result
\[
\frac{265}{153}<\sqrt 3<\frac{1351}{780},
\]
mentioned in `Measurement of a Circle', but a number of possible methods have been suggested. They may be found on the web and elsewhere; see \cite{a,d,heath1,heath2,sond}. A good candidate should satisfy the following criteria.
\begin{itemize}
\item It should only use methods that were known at the time of Archimedes;
\item It should be possible to replace $3$ by any other positive integer;
\item It should be as short and elementary as possible;
\item It should not involve extraordinary ingenuity or tricks.
\end{itemize}
Some of the proposals discussed by Heath in \cite[pp. 51, 52]{heath2} fail the first test, unless some much later authors were describing methods that had been known centuries before their own time.

Section~4 describes a new `interpolation' method that meets all four criteria, and that, if correct, would further confirm Archimedes' understanding of limits. In Section~6 we prove that Archimedes' lower bound cannot be obtained by Hero's method. Section~\ref{pi-evaluation} discusses a number of much harder square roots that Archimedes calculates in the course of obtaining his even more famous upper and lower bounds
\[
3\txtfrac{10}{71} <\pi <3\txtfrac{1}{7}.
\]
Several quite difficult new issues emerge, but it is again likely that Archimedes used an interpolation method to calculate the square roots.

Before starting we mention three standard problems in this type of enterprise. Euclid's Elements survives in a variety of substantially different versions, all of which are copies made hundreds of years after he died. Although he must have written it around $300$ BC, little is known about him personally. For this reason the word Euclid should be regarded as a label rather than a name. See \cite[Chap.~2, 4]{cuomo} and \cite[pp. 354-361]{heath1} for more detailed accounts of this issue. Two comprehensive sources of the Elements are \cite{heath-euclid} and \cite{fitzpatrick}, but a valuable summary may also be found at \cite{joyce}.

Archimedes of Syracuse (c.\ 287-212 BC) is a less shadowy figure than Euclid, but our knowledge of his mathematical work is similarly indirect. His `Measurement of a Circle' survives via numerous copies and translations of a ninth century document written in Constantinople, \cite[pp.25-27]{heath2}. It is also a part of the famous Archimedes palimpsest, transcribed in the tenth century, overwritten in the thirteenth century and discovered in 1906 by Heiberg. The palimpsest is now being studied by means of sophisticated imaging techniques at the Walters Art Museum in Baltimore, \cite{ArchPalProj}.

The second problem is that with the knowledge of later developments, one can easily read more into an ancient text than is there. In particular, it is tempting to assume that an author was aware of what now seems to be an obvious development of an idea in his surviving work. It has, for example, been argued that Euclid and his immediate successors did not understand the fundamental theorem of arithmetic in the way that we now do, \cite{taisbak}. One must constantly be aware of the dangers of falling into such a trap.

The third problem is finding a satisfactory compromise between notation and concepts that Archimedes would have understood and our very different way of expressing the problems of interest. This article uses modern terminology for the calculations, including the free use of negative numbers, but only in circumstances in which there seems to be no problem with rewriting the arguments appropriately.

Throughout the paper Roman letters are used to denote natural numbers and Greek letters to denote fractions or other real numbers. All numerical values are rounded down so that the digits presented are digits of the exact value.

\section{Brute force}
The brute force method involves calculating $m^2$ and $3n^2$ for all $m, \, n$ up to $1351$, and then finding $m,\, n$ such that $m^2-3n^2$ is very small. If one calculates one square per minute, producing such a table would take twenty-four hours, and the table would then be available for other uses. The squares need not be calculated ab initio because of the identity $(n+1)^2=n^2+2n+1$; this has the independent advantage that one can detect occasional errors by examining the result for every $n$ that is a multiple of $10$. Napier and others completed much more arduous tasks before computers were invented. The fact that the method fails the third of our tests is not decisive, in spite of the complexities of Greek arithmetic when dealing with numbers much bigger than a thousand. Archimedes might have used this method, but it would have been much more laborious than the interpolation method that we describe in Section~\ref{IM}.

Table~\ref{tone} shows a fragment of a systematic calculation of the squares of the first thousand natural numbers.
\begin{table}[!h]
\[
\begin{array}{rr}
120^2=14400&\hspace{2em} 3\times 120^2=43200\\
+241&+723\\
121^2=14641&3\times 121^2=43923\\
+243&+729\\
122^2=14884&3\times 122^2=44652\\
+245&+735\\
123^2=15129&3\times 123^2=45387\\
+247&+741\\
124^2=15376&3\times 124^2=46128
\end{array}
\]
\caption{Calculations of squares\label{tone}}
\end{table}

Table~\ref{ttwo} presents a selection of the results of the calculations indicated above. In the second column the upper bounds to $\sqrt{3}$ are displaced slightly to the right.
\begin{table}[!h]
\[
\begin{array}{rl}
5^2- 3\times 3^2=-2,\hspace{2em}&5/3=1.6,\\
7^2-3\times  4^2=1,\hspace{2em}&\hspace{1.5em}7/4=1.75,\\
19^2- 3\times  11^2= -2,\hspace{2em}&19/11\sim 1.727272,\\
26^2- 3\times  15^2=1,\hspace{2em}&\hspace{1.5em}26/15\sim 1.733333,\\
71^2- 3\times  41^2= -2,\hspace{2em}&71/41\sim 1.731707,\\
97^2- 3\times  56^2=1,\hspace{2em}&\hspace{1.5em}97/56\sim 1.732142,\\
265^2- 3\times  153^2= -2,\hspace{2em}&265/153\sim 1.732026,\\
362^2-3\times  209^2=1,\hspace{2em}& \hspace{1.5em}362/209\sim 1.732057,\\
989^2- 3\times  571^2= -2,\hspace{2em}& 989/571\sim 1.732049,\\
1351^2- 3\times  780^2=1,\hspace{2em}&\hspace{1.5em}1351/780\sim 1.732051.
\end{array}
\]
\caption{Results of brute force calculations\label{ttwo}}
\end{table}

\section{Pell's equation}\label{PE}

It is often suggested that Archimedes might have solved the equations $a^2-3b^2=m$ for $m=1$ and $m=-2$ systematically; it should be mentioned here that $a^2-3b^2=-1$ has no integer solutions. We argue that this is not plausible on historical grounds.

The equation $a^2-cb^2=m$ is often called Pell's equation, particularly when $m=1$, because of an unfortunate error; it has nothing to do with Pell. It is obviously equivalent to $(a/b)^2-c=m/b^2$, so any integer solution with $b$ large and $m$ small yields an accurate rational approximation to $\sqrt{c}$.
This method is indeed one of the best ways of approximating $\sqrt{c}$, and the main problem is finding evidence that Archimedes would have been able to solve the equation when $c=3$. The situation is as follows.

It is known from cuneiform clay tablets that the Babylonians were able to calculate square roots accurately long before the time of Archimedes, but there is no contemporary evidence of how they did this. In the fifth century AD Proclus claimed that the Pythagoreans calculated the square root of $2$ by using the following fact, which I have rewritten in modern notation. If $a_n,\, b_n$ satisfy $a_n^2-2b_n^2=\pm 1$ and one puts $a_{n+1}=a_n+2b_n$, $b_{n+1}=a_n+b_n$ then $a_{n+1}^2-2b_{n+1}^2=\mp 1$. Starting from $a_1=b_1=1$ this generates a sequence of steadily improving upper and lower bounds alternately. This is, in fact, a systematic procedure for solving Pell's equation when $c=2$ and $m=\pm 1$.

Proclus's claim may have been based on guesswork or on documentary sources that are now lost. One has to be careful not to accept it uncritically, because he would have known the much later work on such equations by Diophantus of Alexandria, written in the third century AD. However he quotes the identity
\[
(a+2b)^2-2(a+b)^2=-(a^2-2b^2)
\]
from the Elements, Book~2, Prop.~10. This makes a reasonable case for believing Proclus, but it does not imply that Archimedes was able to generalize the procedure to $c=3$. Indeed there is no evidence to support this. Solving such equations might now seem elementary, but the evidence is that it took many hundreds of years before the relevant theory became routine knowledge; see \cite[Chap.~5B]{VDW}. Some of the identities needed look very artificial if one does not express them in term of quadratic arithmetic in the ring $\Z(\sqrt{c})$, within which the equation can be rewritten in the form $(a+b\sqrt{c})(a-b\sqrt{c})=m$.

\section{An interpolation method}\label{IM}
We describe the method in modern notation for the convenience of the reader, and explain its relationship with the use of continued fractions in the next section. The method differs from others in that it starts with two approximations rather than one, and passes by an elementary calculation to a new pair of approximations.

The only result that we use is the fact that
\begin{equation}
\frac{v}{u}<\frac{y}{x}\hspace{2em}\mbox{ implies }%
\hspace{2em}\frac{v}{u}<\frac{v+y}{u+x}<\frac{y}{x}.\label{two}
\end{equation}
We have not found this in the Elements, but it can be deduced from Book~5, Prop.~12, whose proof is logical rather than geometric. It may also be proved by modifying the proof of Prop.~12 appropriately.

One can obtain an insight into (\ref{two}) by considering the slope of the diagonal of a parallelogram that has two adjacent edges with slopes $v/u$ and $y/x$ as in Figure~\ref{interpolfigure}. It is not obvious whether one can convert this use of coordinate geometry into a form that Euclid or Archimedes could have understood.
\begin{figure}[!h]
\begin{center}
\scalebox{0.9}{
\unitlength=1ex
\begin{picture}(40,35)
\put(0,0){\vector(0,1){35}}
\put(0,0){\vector(1,0){40}}
\put(0,0){\line(4,1){20}}
\put(0,0){\line(1,5){5}}
\put(0,0){\line(5,6){25}}
\put(5,25){\line(4,1){20}}
\put(20,5){\line(1,5){5}}
\put(23,5){\makebox(0,0){$(u,v)$}}
\put(5,27){\makebox(0,0){$(x,y)$}}
\put(32,30){\makebox(0,0){$(u+x,v+y)$}}
\end{picture}
} 
\end{center}\vspace{0ex}
\caption{Interpolation using slopes\label{interpolfigure}}
\end{figure}
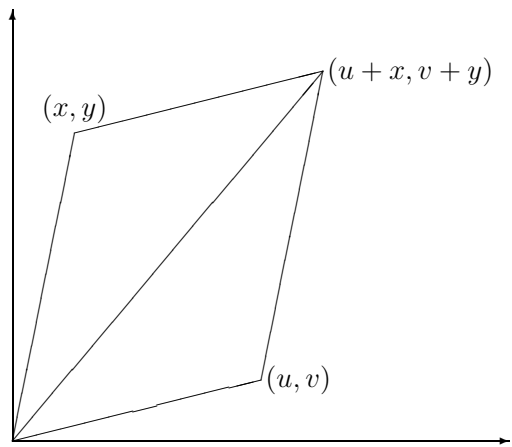

An alternative proof of (\ref{two}) involves rewriting the hypothesis in the form $vx<yu$ and interpreting this as comparing the areas of two rectangles. The first inequality in the conclusion is rewritten in the form $v(u+x)<(v+y)u$ and is also interpreted in terms of areas. The result is then obvious by inspecting Figure~\ref{interpfig2}. The second inequality has a similar proof. This type of transformation is to be found in the Elements, Book~6, Prop.~14, and (\ref{two}) can be proved from Prop.~14.

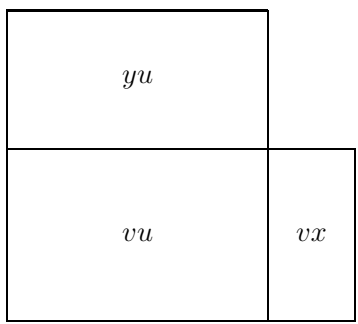
\begin{figure}[!h]
\begin{center}
\scalebox{0.9}{
\unitlength=0.7ex
\begin{picture}(40,36)
\put(0,0){\line(0,1){36}}
\put(0,0){\line(1,0){40}}
\put(0,20){\line(1,0){40}}
\put(0,36){\line(1,0){30}}
\put(30,0){\line(0,1){36}}
\put(40,0){\line(0,1){20}}
\put(15,10){\makebox(0,0){$vu$}}
\put(15,28){\makebox(0,0){$yu$}}
\put(35,10){\makebox(0,0){$vx$}}
\end{picture}
} 
\end{center}\vspace{0ex}
\caption{Interpolation using areas\label{interpfig2}}
\end{figure}

Given two fractions $\alp=v/u$ and $\bet=y/x$ such that $\alp<\sqrt{3}<\bet$, one puts
\[
\gam=\frac{v+y}{u+x}.
\]
One then replaces $\alp$ by $\gam$ if $\gam<\sqrt{3}$, or equivalently if $(v+y)^2<3(u+x)^2$, and replaces $\bet$ by $\gam$ if $\gam >\sqrt{3}$, or equivalently if $(v+y)^2>3(u+x)^2$, to obtain a new and better enclosure of $\sqrt{3}$.  This is repeated until the desired accuracy is achieved.

The basic idea of this method, obtaining steadily better upper and lower bounds on a number by a systematic algorithm, was used by Archimedes on many occasions. See, for example, Section~\ref{pi-evaluation}.

A list of the first sixteen iterations of the interpolation method starting from the pair $1,\, 2$ is shown in Table~\ref{tthree}.
\begin{table}[!h]
\[
\begin{array}{ccccc}
\frac{1}{1}<\sqrt{3}<\frac{2}{1}&\hspace{1em}&%
\frac{3}{2}<\sqrt{3}<\frac{2}{1}&\hspace{1em}&%
\frac{5}{3}<\sqrt{3}<\frac{2}{1}\vspace{1ex}\\
\frac{5}{3}<\sqrt{3}<\frac{7}{4}&\hspace{1em}&%
\frac{12}{7}<\sqrt{3}<\frac{7}{4}&\hspace{1em}&%
\frac{19}{11}<\sqrt{3}<\frac{7}{4}\vspace{1ex}\\
\frac{19}{11}<\sqrt{3}<\frac{26}{15}&\hspace{1em}&%
\frac{45}{26}<\sqrt{3}<\frac{26}{15}&\hspace{1em}&%
\frac{71}{41}<\sqrt{3}<\frac{26}{15}\vspace{1ex}\\
\frac{71}{41}<\sqrt{3}<\frac{97}{56}&\hspace{1em}&%
\frac{168}{97}<\sqrt{3}<\frac{97}{56}&\hspace{1em}&%
\frac{265}{153}<\sqrt{3}<\frac{97}{56}\vspace{1ex}\\
\frac{265}{153}<\sqrt{3}<\frac{362}{209}&\hspace{1em}&%
\frac{627}{362}<\sqrt{3}<\frac{362}{209}&\hspace{1em}&%
\frac{989}{571}<\sqrt{3}<\frac{362}{209}\vspace{1ex}\\
\frac{989}{571}<\sqrt{3}<\frac{1351}{780}.&&
\end{array}
\]
\caption{The interpolation method\label{tthree}}
\end{table}
It may be seen that Archimedes' upper and lower bounds are obtained by this method. However, his lower bound is by no means the best in the list. There is not enough information to explain this, but the following is plausible. He might initially have stopped calculating at
\[
\frac{265}{153}<\sqrt{3}<\frac{362}{209}.
\]
If so, he would later have discovered that the upper bound on $\sqrt{3}$ was not strong enough to obtain the lower bound that he desired on $\pi$ and then proceeded further with the iteration above, but only recorded the improved \emph{upper} bound to $\sqrt{3}$, because he did not want to repeat more calculations than necessary.

The interpolation method is systematic and completely elementary, apart from depending on being able to determine whether a fraction $a/b$ is greater than or less than $\sqrt{3}$. This is equivalent to whether $a^2-3b^2$ is positive or negative. This can be decided by evaluating the squares involved, which is quite laborious when $a$ and $b$ are large. However, mathematicians of all eras up to the invention of pocket calculators were very skilled when carrying out arithmetic calculations.

Dominus has described a method that yields the same upper and lower bounds as those obtained by the interpolation method, \cite{d}. It depends on producing a list of good approximations to $\sqrt{3}$ with small numerators and denominators and then finding a pattern that can be extrapolated. It is finally necessary to check that the extrapolated values are indeed good approximations. The main problem with this method is that if one applies it to more complex problems, such as the determination of the square roots of larger numbers, no obvious pattern is likely to emerge. On the other hand the interpolation method always works because of its obvious similarity to the interval bisection proof of the intermediate value theorem.

\section{Continued fractions}\label{CF}

The fraction $a/b$ can only be close to $\sqrt{3}$ if
$m=a^2-3b^2$ is small relative to $b^2$. This proves that it should
be possible to find a connection between any method of approximating
$\sqrt{3}$ by fractions and solutions of Pell's equation or some quadratic Diophantine generalization. Pell's equation is
known to be connected in turn to the theory of continued fractions, \cite{williams}. In this section we show that the method of continued fractions is a special case of the interpolation method. This by no means implies that the theory of continued fractions underlies the calculations of the last section, which are elementary and rigorous in their own right. There is no evidence that the use of continued fraction expansions was known in classical Greece, in spite of the fact that determining the highest common factor of two numbers by using the Euclidean algorithm can be interpreted that way.

Instead of providing a complete analysis, we consider the particular case in which some positive real number $x$ is approximated by
\[
x\sim a_1+\frac{1}{a_2+\frac{1}{a_3+\frac{1}{a_4}}}
\]
where $a_r$ are all positive integers. Some elementary algebra yields
\begin{equation}
a_1+\frac{1}{a_2+\frac{1}{a_3+\frac{1}{a_4}}}=\frac{\alp+\bet a_4}{\gam+\del a_4}\label{cf1}
\end{equation}
where $\alp,\, \bet,\, \gam,\, \del$ are certain polynomials in $a_1,\, a_2,\, a_3$. Moreover
\begin{equation}
\frac{\alp}{\gam}=a_1+\frac{1}{a_2}\label{cf2}
\end{equation}
and
\begin{equation}
\frac{\bet}{\del}=a_1+\frac{1}{a_2+\frac{1}{a_3}}.\label{cf3}
\end{equation}
The validity of (\ref{cf1}) follows by considering the map
\[
a_4\to a_1+\frac{1}{a_2+\frac{1}{a_3+\frac{1}{a_4}}}
\]
as the composition of a sequence of fractional linear transformations on $\R\cup\{ \infty\}$. The equations (\ref{cf2}) and (\ref{cf3}) can then be obtained by putting $a_4=0,\, \infty$ respectively. These formulae show that if the positive integers $a_1,\, a_2,\, a_3$ have already been chosen by the continued fraction method to approximate $x$ as well as possible, then the choice of the next positive integer $a_4$ is of exactly the type discussed in Section~\ref{IM}.

\section{Hero's method}\label{HM}

In the first century AD Hero of Alexandria described a method which converges more rapidly than the interpolation method. Perhaps over-generously, we will use the term `Hero's method' to refer to any algorithm for calculating square roots that relies upon the fact that the arithmetic mean of two numbers is greater than their geometric mean, which in turn is larger than their harmonic mean. The AM-GM inequality is mathematically equivalent to the GM-HM inequality and is a special case of the Elements, Book~5, Prop.~25. It may be proved by observing that one can put four disjoint rectangles of size $a\times b$ inside a square with edge length $a+b$, as in Figure~\ref{amgmfigure}. These means were known to the Pythagoreans and were mentioned in Plato's Timaeus 35, 36.
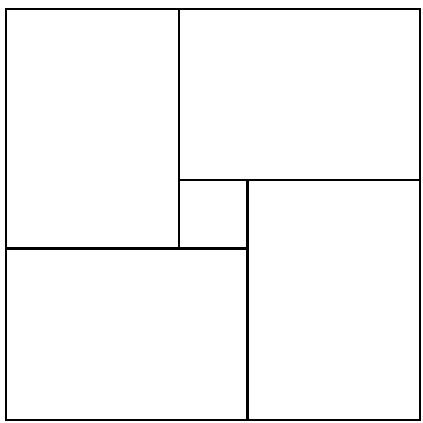
\begin{figure}[!h]
\begin{center}
\unitlength=0.5ex
\begin{picture}(100,80)(0,20)
\put(20,20){\line(0,1){60}}
\put(80,20){\line(0,1){60}}
\put(20,80){\line(1,0){60}}
\put(20,20){\line(1,0){60}}
\put(55,20){\line(0,1){35}}
\put(80,55){\line(-1,0){35}}
\put(45,80){\line(0,-1){35}}
\put(20,45){\line(1,0){35}}
\end{picture}
\end{center}
\caption{The AM-GM inequality\label{amgmfigure}}
\end{figure}

Hero's original method may be found in his Metrica, which was lost for many centuries, but eventually rediscovered in 1896 in Constantinople by R Sch\"one. An account of Hero's mathematical work may be found in \cite[Chap. 18]{heath2} and other important contributions of his are described in \cite{Pap}. Applying Hero's method to the numbers $a$ and $b=3/a$, whose geometric mean is $\sqrt{3}$, one obtains
\[
a^\pr=\frac{a+3/a}{2}>\sqrt{3}.
\]
Starting from $a=5/3$ this can be applied twice to yield
\[
\frac{26}{15}>\sqrt{3}\mbox{  and then  } \frac{1351}{780}>\sqrt{3}.
\]
In spite of this, it is not possible to use arithmetic and harmonic means to obtain Archimedes' lower bound $\frac{265}{153}$, because of the following theorem; note that $3$ is not a factor of $265$. This proves that Archimedes did not use Hero's method to obtain his lower bound.

\begin{theorem}
Let $x,\, y$ be two distinct positive rational numbers with $xy=3$ and let $a,\, h$ be their arithmetic and harmonic means respectively. Then $ah=3$ and $h<\sqrt{3}<a$.  Moreover $3$ must be a factor of the numerator of $h$.
\end{theorem}

\begin{proof}
The formulae
\[
a=\frac{x+3/x}{2}, \hspace{2em}h=\frac{2}{1/x+x/3}
\]
yield $ah=3$. The second statement of the theorem follows from the fact that the arithmetic mean of two numbers is greater than their geometric mean, which is $3$ in this case, while their harmonic mean is smaller.

In order to prove the final statement we put $x=n/m$ where $m,\, n$ are relatively prime, and consider the resulting expression
\[
h=\frac{6mn}{3m^2+n^2},
\]
in which the numerator and denominator may have some common factors. One needs to consider the effect of removing these. If $3$ is not a factor of $n$ then it is not a factor of the denominator, so the $3$ is the numerator cannot cancel out. If $3$ is a factor of $n$ then it cannot be a factor of $m$, so $3$ is a factor of the denominator but $3^2$ is not. On the other hand $3^2$ is a factor of the numerator, so after reducing $h$ to its lowest terms, the numerator must retain a factor $3$.
\end{proof}

In his account of Hero's work, Heath describes a complicated mixed strategy for obtaining Archimedes' lower bound, %
\cite[pp. 324, 325]{heath2}. He first applies Hero's method starting from $a=1$ to obtain
\begin{equation}
\frac{97}{56}<\sqrt{3}<\frac{168}{97}\label{th}
\end{equation}
in three steps. Further iterations of Hero's method from this starting point do not yield Archimedes' upper or lower bounds. Then he observes that if one applies the interpolation method to the two fractions in (\ref{th}) one obtains
\[
\frac{97+168}{56+97}=\frac{265}{153}
\]
and leaves the reader to infer that this is a lower bound. Finally he re-applies Hero's method starting from $a=5/3$ to obtain Archimedes' upper bound. Although he does eventually obtain the result required, he does so by using Hero's method with two different starting points and provides no reason why one should then apply the interpolation method to the particular pair of fractions in (\ref{th}) rather than to some other pair. It therefore fails to satisfy our third and fourth criteria for a satisfactory solution to the problem.

\section{Upper and lower bounds on $\pi$} \label{pi-evaluation}

Archimedes' upper and lower bounds
\[
3\txtfrac{10}{71} <\pi <3\txtfrac{1}{7}
\]
on $\pi$ have been discussed in some detail by Heath, \cite[pp.~50-56]{heath2}. We follow his account of Archimedes' text, but draw significantly different conclusions about the often unstated numerical methods that Archimedes used at various stages in the calculation.

Archimedes' upper and lower bounds were obtained by approximating the edge lengths of regular polygons that have $6,\, 12,\, 24,\, 48,$ and finally $96$ edges. There are three reasons for his residual error: he started with upper and lower bounds to $\sqrt{3}$ rather than the exact value, at each stage he obtained upper or lower bounds on certain further square roots, and he stopped at the $96$-gon stage. Numerically his final result is
\[
 3.140845 < \pi < 3.142857\, .
\]
With current computational facilities, one can assess the effect of the inexact values of Archimedes' square roots as follows. The exact bounds obtained by calculating the total perimeter of the $96$-gon are
\[
96\sin(\pi/96)<\pi<96\tan(\pi/96),
\]
which translate into the numerical bounds
\[
3.141031 <\pi<3.142714\, .
\]
By Taylor's theorem, each further doubling of the number of sides reduces both errors by a factor of $4$. At each stage the error in the upper bound is about twice as big as the error in the lower bound -- until the rounding errors become important.

Archimedes relied on the following theorem to pass from one polygon to another with twice the number of sides.

\begin{theorem}[Elements, Book~1, Props.~5 and 32]\label{pythaginduct}
Let $\theta_n$ be the internal angle of the right-angled triangle with edge lengths $a_n$, $c_n$ and $b_n=\sqrt{a_n^2+c_n^2}$, as shown in Figure~\ref{halving}. Then the triangle with edge lengths \[
c_{n+1}=c_n,\hspace{2ex} a_{n+1}=a_n+b_n,\hspace{2ex} b_{n+1}=\sqrt{a_{n+1}^2+c_{n+1}^2}
\]
has internal angle $\theta_{n+1}=\theta_n/2$.
\end{theorem}

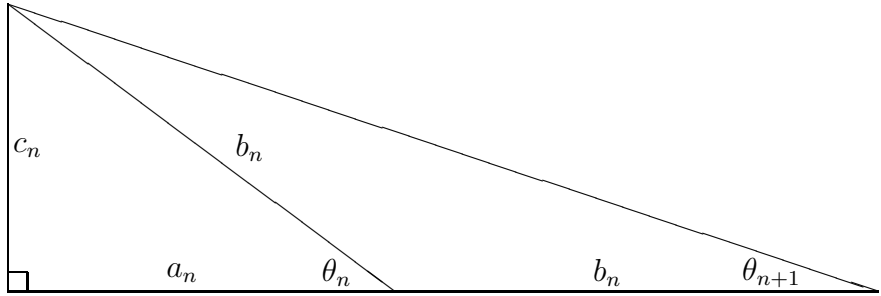
\begin{figure}[!h]
\begin{center}
\unitlength=0.7ex
\begin{picture}(90,30)
\put(0,0){\line(1,0){90}}
\put(0,0){\line(0,1){30}}
\put(0,30){\line(4,-3){40}}
\put(0,30){\line(3,-1){90}}
\put(2,0){\line(0,1){2}}
\put(0,2){\line(1,0){2}}
\put(18,2){\makebox(0,0){$a_n$}}
\put(2,15){\makebox(0,0){$c_n$}}
\put(62,2){\makebox(0,0){$b_n$}}
\put(79,2){\makebox(0,0){$\theta_{n+1}$}}
\put(25,15){\makebox(0,0){$b_n$}}
\put(34,2){\makebox(0,0){$\theta_{n}$}}
\end{picture}
\end{center}
\caption{Halving the internal angle\label{halving}}
\end{figure}

Applying Theorem~\ref{pythaginduct} repeatedly, Archimedes had to evaluate seven square roots, in addition to $\sqrt{3}$, but did not explain how he did this, \cite[pp.~55, 56]{heath2}.  Five of the seven square roots are approximated, above or below as necessary, by an integer plus a fraction whose denominator is $4$ or $8$. Any explanation of how he performed the calculations must take this fact into account.  It is probably relevant that dyadic fractions had already been used in measurement by the ancient Egyptians. One of his bounds was
\[
\gam=\sqrt{571^2+153^2}>591\txtfrac{1}{8}.
\]
A numerical computation of the LHS yields $\gam\sim 591.14296$, which confirms his result and indicates its level of accuracy.

It is likely that Archimedes proceeded as follows. First he found the integer part $n$ of $\gam$. This could have been done by using Pythagoras' theorem graphically, i.e.\, by drawing the right angled triangle whose two shorter edges had lengths $571$ and $153$. This produces the approximate value $590$ for the hypotenuse immediately, after which a few calculations yield $591<\gam<592$. He then applied the following interpolation method to obtain more accurate upper and lower bounds.
He bisected the interval $[591,592]$ and chose the subinterval containing $\gam$ by calculating the square of the midpoint; after repeating this procedure three times, he chose the left hand end of the resulting subinterval because he needed a lower bound to $\gam$. If this is how he proceeded, he would have obtained the results in Table~\ref{tfour}, which yields
\begin{table}[!t]
\begin{eqnarray*}
\gam^2=571^2+153^2&=& 349450\, ,\\
590^2&=&348100 \, ,\\
592^2&=&350464\, ,\\
591^2&=&349281\, ,\\
(591\txtfrac{1}{2})^2&=& 349872\txtfrac{1}{4}\, ,\\
(591\txtfrac{1}{4})^2&=& 349576\txtfrac{9}{16}\, ,\\
(591\txtfrac{1}{8})^2&=& 349428\txtfrac{49}{64}\, .
\end{eqnarray*}
\caption{Calculations to determine $\gam$ \label{tfour}}
\end{table}
$(591\txtfrac{1}{8})^2<571^2+153^2$, or equivalently $591\txtfrac{1}{8}<\gam$.

The squares in the above interpolation could have been calculated directly, but starting from the fourth line Archimedes might have used the formula
\begin{equation}
\left( \frac{a+b}{2}\right)^2 = \frac{a^2+b^2}{2}-\left( \frac{a-b}{2}\right)^2,\label{book2.9}
\end{equation}
which would have reduced his labour considerably, because at each stage $a^2,\, b^2$ and $a-b$ are all known. The formula (\ref{book2.9}) is contained in the Elements, Book 2, Prop.~9 or 10.

We next discuss the first of the two exceptional cases, in which Archimedes recorded the inequality
\[
\alp=\sqrt{1823^2+240^2}<1838\txtfrac{9}{11}.
\]
A numerical computation yields $\alp\sim 1838.7302$, so according to the above method Archimedes should have recorded the better upper bound $1838\txtfrac{3}{4}$. It is plausible that he obtained this value by the interpolation method described above, and then adjusted it at the next step, where he used the simplification
\[
\frac{1823+1838\txtfrac{9}{11}}{240}=\frac{3661\txtfrac{9}{11}}{240}=\frac{1007}{66}.
\]
This sufficed for his purposes and would have reduced the labour involved in later calculations. Unfortunately the fraction
\[
\frac{3661\txtfrac{3}{4}}{240}
\]
cannot be simplified.

The second exceptional case must have a different explanation. Archimedes obtained
\begin{equation}
\bet=\sqrt{ 1007^2+66^2}<1009\txtfrac{1}{6}\label{upperbet}
\end{equation}
which led him after one further calculation to
\begin{equation}
\pi>\frac{96\times 66}{2017\txtfrac{1}{4}}>3\txtfrac{10}{71} \hspace{3ex} (\sim 3.140845).\label{lowerpi}
\end{equation}

It is suggested that he originally proved that $\bet< 1009\txtfrac{1}{4}$ by the interpolation method, and then continued to obtain
\[
\pi>\frac{96\times 66}{2017\txtfrac{1}{3}}>3\txtfrac{9}{64} \hspace{3ex} (= 3.140625).
\]
Not satisfied with this, he modified the interpolation method slightly to obtain (\ref{upperbet}) and then the slightly better lower bound (\ref{lowerpi}) on $\pi$.

\textbf{Acknowledgements} I should like to thank Y. N. Petridis, J. R. Silvester and D. R. Solomon for helpful comments.

Department of Mathematics\\
King's College London\\
Strand\\
London WC2R 2LS\\
UK

E.Brian.Davies@kcl.ac.uk

\end{document}